\documentclass[12pt, letterpaper]{article} 
\usepackage{amsmath, amsthm, amssymb, wasysym, verbatim, bbm, color, graphics, geometry, mathrsfs, tipa, bm}
\usepackage{dcolumn}
\usepackage{tikz}
\usepackage{tikz-cd}
\usepackage{graphicx}

\graphicspath{ {./images/} }

\newcommand{\R}{\mathbb{R}}
\newcommand{\C}{\mathbb{C}}
\newcommand{\Z}{\mathbb{Z}}
\newcommand{\N}{\mathbb{N}}

\tikzcdset{every label/.append style = {font = \normalsize}}

\refstepcounter{equation}

\newtheorem{lem}{Lemma}
\newtheorem{thm}[lem]{Theorem}

\newtheorem{cor}[lem]{Corollary}

\newtheorem{conj}{Conjecture}

\theoremstyle{definition}
\newtheorem{definition}[lem]{Definition}

\renewcommand{\descriptionlabel}[1]%
     {\hspace{\labelsep}\textsf{#1}}

\begin{document}

\title{Similarity at Misiurewicz Maps in
the Cubic Parameter Curves}
\author{Araceli Bonifant and Brady Young}
\date{}

\maketitle
\begin{abstract}
We present a proof of the conjecture by Bonifant and Milnor (see \cite{CM3}) regarding the similarity between the connectedness locus of the curve $\mathcal{S}_p$ at Misiurewicz parameters and their corresponding filled Julia sets in a neighborhood of the corresponding free co-critical point. The proof is in parallel with the generalization of Tan Lei's proof of similarity in the Mandelbrot set developed by Kawahira.
\end{abstract}

\vspace{.25in}

\setcounter{equation}{0}
\setcounter{lem}{0}
\section{Introduction}

A notable result in the study of quadratic polynomials is the similarity at Misiurewicz points in the Mandelbrot set $\mathbb{M}$.  In essence, there is a 
dense set of points $c$ on the boundary of $\mathbb{M}$ where the structure of the boundary, zooming to a very high scale into a very small neighborhood of this point, is very similar to the structure in a neighborhood of $c$ in its associated Julia set.  Initially proven by Tan Lei in \cite{TL90:1990}, new techniques for proving this result were introduced by Kawahira in \cite{Kawa20:2020}.

\begin{figure}[h]
    \centering
    \includegraphics[scale=.225]{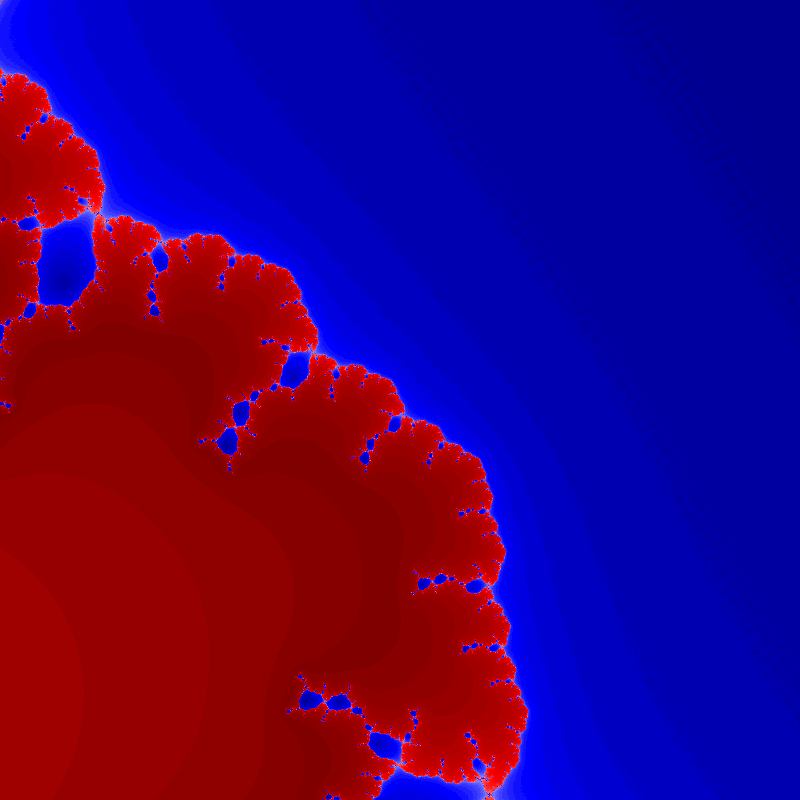}
    \includegraphics[scale=.225]{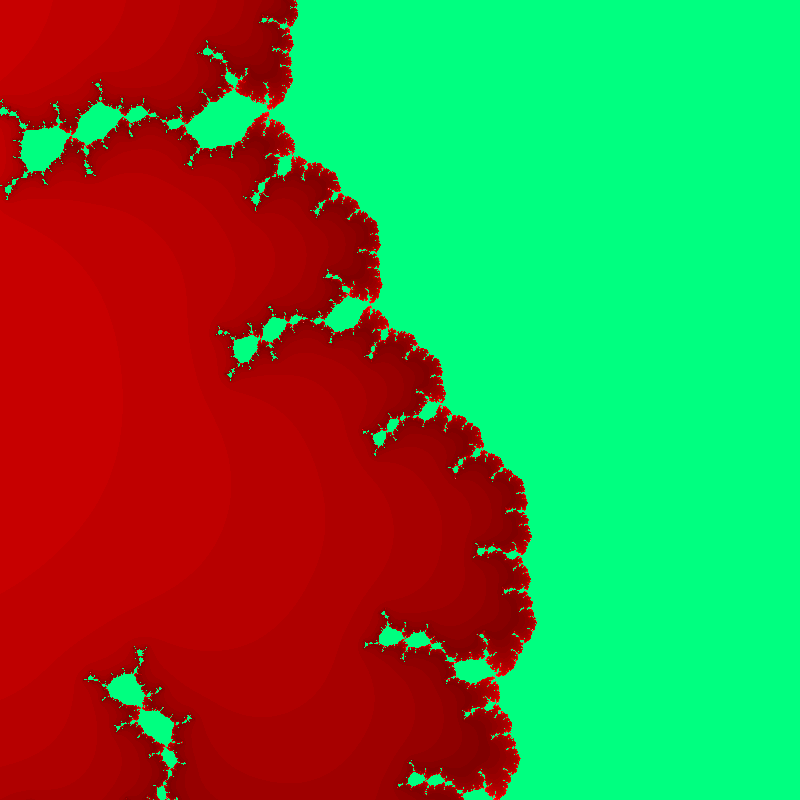}
    \caption{\textbf{Left}: \textsf{The parameter curve $\mathcal{S}_2$ near a Misiurewicz map. \textbf{Right}: The filled Julia set of the Misiurewicz map near the co-critical point $2a$}}
    \label{fig:enter}
\end{figure}

Milnor\footnote{Milnor's paper was circulated circa 1992 and published in 2009.}, and Branner and Hubbard developed a theory of the parameter space of cubic polynomials, see \cite{CM1:2009} and \cite{BH}. Milnor introduced the superattracting period $p$ cubic parameter curves known as $\mathcal{S}_p$, a way of taking one dimensional slices of the two dimensional cubic parameter space. Following Milnor, every cubic polynomial can written in the normal form
\begin{equation}\label{e-0}
F_{a,v}(z) = z^3 - 3a^2z + (2a^3+v)~.
\end{equation}
For a map in $\mathcal{S}_p$, the marked critical point $a_F$ is periodic of period $p$. 

In these curves, the \textbf{Misiurewicz maps} are maps where the free critical point $-a_F$ is repelling strictly preperiodic. It was observed in \cite{CM3} that the Riemann surface $\mathcal{S}_p$ exhibits the so-called \textbf{Tan-Lei similarity} at Misiurewicz maps,\begin{footnote}{
It was noted in \cite{SCP} that following Tan Lei's  method of proof for the quadratic case, one could obtain a similarity result between the connectedness locus of cubic polynomials at a ``more general notion'' of Misiurewicz map and their associated Julia set, but this would require projections of the parameter space that may alter the topology of the curve in the case of $\mathcal{S}_p$.}
\end{footnote}and the following was conjectured:
\begin{conj}
Under iterated magnification, the connectedness locus around a Misiurewicz map $F \in \mathcal{S}_p$ looks more and more like the filled Julia set $K_F$ of $F$ around the free co-critical point $2a_F$ (up to a fixed scale change and rotation).
\end{conj}
The important tool used to study the parameter curves in this respect are the \textbf{local coordinates} developed in \cite{CM2:2010}, which give a local parametrization of $\mathcal{S}_p$ in open neighborhoods. Using the idea in \cite{Kawa20:2020}, as well as the restriction to locality, the conjecture is proven and is presented as:
\begin{thm}[\bf Main Theorem]\label{t-main}
Let $F \in \mathcal{S}_p$ be a Misiurewicz map. Assume that the curve $\mathcal{S}_p$ has been parametrized locally near $F$ by a parameter $\textbf{t}$. Then there exists a non-constant entire function $\phi$ on $\C$, a sequence $\rho_k \to 0$, and a constant $q \neq 0$ such that if we set $\mathscr{K} = \phi^{-1}(K_F) \subset \C$, then for any large $r > 0$, we have that
\begin{enumerate}
\item
$[\rho_k^{-1}(K_F - 2a_F)]_r \to [\mathscr{K}]_r$
\item
$[\rho_k^{-1}q(\mathscr{C}(\mathcal{S}_p) - F)]_r \to [\mathscr{K}]_r$
\end{enumerate}
as $k \to \infty$ in the Hausdorff topology. 
\end{thm}
In the theorem, $\mathscr{C}(\mathcal{S}_p)$ refers to the connectedness locus of the curve $\mathcal{S}_p$, and  the expression $\mathscr{C}(\mathcal{S}_p) - F$ indicates that we will center the magnification at the given Misiurewicz map $F$. We will set the map $F = F_{\bm{t}=0}$ using the  local parametrization. For any closed set $A\subset {\mathbb C}$, the compact set $[A]_r$ is defined as: 
\[
[A]_r = \big(\overline{\mathbb{D}(r)} \cap A\big) \cup \partial \mathbb{D}(r)~.
\]

\setcounter{equation}{0}
\setcounter{lem}{0}
\section{Similarity and The Quadratic Case}
Throughout this note, we use the \textbf{Hausdorff topology} on the non-empty compact subsets of $\C$, endowed by the Hausdorff distance between compact sets $A$ and $B$
\[
d(A,B) = \sup(\delta(A,B), \, \delta(B,A))
\]
where
\[
\delta(A,B) = \sup_{x \in A} d(x,B)
\]

We will require a notion of convergence in this topology. A sequence of compact sets $E_k$ \textbf{converges} to a limit set $E$ if for all $\epsilon > 0$, there exists $k_0 \in \N$ such that $E \subset N_\epsilon(E_k)$ and $E_k \subset N_\epsilon(E)$ for all $k \geq k_0$. Here $N_\epsilon(\cdot)$ denotes the open $\epsilon$-neighborhood of a given compact set in $\C$. In such a case, we will write that $E_k \to E$.

\begin{definition}\label{d-1}
A closed set $A \subset \C$ is \textbf{asymptotically $q$-self-similar} about a point $c \in \C$ if there is $r \in \R^+$ and a closed set $B$ such that
\[
[q^k(A-c)]_r \to [B]_r
\]
as $k \to \infty$ in the Hausdorff topology.  The set $B$ is automatically $q$-self-similar about $0$, and it is called the limit model of $A$ at $c$.

Two closed sets $A$ and $B$ are \textbf{asymptotically similar} about $0$ if there is $r \in \R^+$ such that in the Hausdorff-Chabauty distance in the window ${\mathbb D}_r$,
\[
\lim_{j \in \C,\, j \to \infty} d([jA]_r,\,[jB]_r)  = 0~.
\]
\end{definition}

In \cite[Proposition 2.2]{TL90:1990} the author proved that a sufficient condition  for the asymptotic similarity about $0$ of two closed subsets $A$ and $B$  of ${\mathbb C}$ is the existence of  $r \in \R^+$ and $q \in {\mathbb C}$ with $|q| > 1$ such that
\[
\lim_{n\to\infty} d([q^nA]_r, \,[q^nB]_r) = 0~.
\]

It follows from this viewpoint that if the two sequences of closed sets in Theorem~\ref{t-main} approach the same limit set as $n \to \infty$, then they are asymptotically similar about $0$; and so both sets, the connectedness locus $\mathscr{C}(\mathcal{S}_p)$ of $\mathcal{S}_p$, and the filled Julia set  $K_F$ of $F$, exhibit asymptotic self-similarity about the respective points. The method of proof for the main theorem consists in showing that a particular set associated with the parameter space and a set associated with the Julia set are both asymptotically $q$-self-similar to the same limit set, which will be defined as a particular preimage of the filled Julia set. This will demonstrate the claimed similarity.

\vspace{.25in}

For a quadratic map $f_c: z \mapsto z^2+c$, the \textbf{filled Julia set} $K_c$ of $f_c$ is the set of non-escaping points under iteration
\[
K_c = \big\{z \in \C : \,\text{$\{f_c^n(z)\}$ is bounded}\big\}
\]
and the \textbf{Julia set of $f_c$},  written $J_c$, is the boundary of $K_c$, so $J_c = \partial K_c$.

The \textbf{Mandelbrot set} $\mathbb{M}$ is the connectedness locus of the quadratic family, the collection of quadratic maps with connected Julia sets. It can also be described as the set of parameters $c \in \C$  such that the critical point $0$ is in the filled Julia set for $z^2+c$, as in
\[
\mathbb{M} = \{c \in \C :\, 0 \in K_c\}~.
\]

A \textbf{Misiurewicz point} is a parameter $c_0 \in \mathbb{M}$ such that the critical point $0$ is not a periodic point, but its orbit eventually enters a repelling periodic orbit. Formally, a point $c_0 \in \mathbb{M}$ is a Misiurewicz point if there exists minimal $\ell \geq 1$ and  $p \geq 1$ such that 
\[
f_{c_0}^\ell(0) = a_0
\]
and
\[
f_{c_0}^p(a_0) = a_0
\]
with $|(f_{c_0}^p)'(a_0)| > 1$.
The similarity of $\mathbb{M}$ and $J_c$ at the Misiurewicz point $c$ was proven in \cite{TL90:1990}, and a different proof of Tan Lei's similarity theorem was given about 30 years later in \cite{Kawa20:2020}. 

\begin{thm}[Tan Lei, 1990]
Let $c_0 \in \mathbb{M}$ be a Misiurewicz point.
There exists a non-constant entire function $\phi$ on $\C$, a sequence $\rho_k \to 0$, and a constant $q \neq 0$ such that if we set $\mathscr{J} = \phi^{-1}(J_{c_0}) \subset \C$, then for any large constant $r > 0$, we have
\begin{enumerate}
\item 
$[\rho_k^{-1}(J_{c_0} - c_o)]_r \to [\mathscr{J}]_r$
\item
$[\rho_k^{-1}q(\mathbb{M} - c_0)]_r \to [\mathscr{J}]_r$
\end{enumerate}
as $k \to \infty$ in the Hausdorff topology.
\end{thm}

\setcounter{equation}{0}
\setcounter{lem}{0}
\section{The Curve $\mathcal{S}_p$}
The curve $\mathcal{S}_p$, introduced in \cite{CM1:2009}, is the space of all monic centered cubic polynomial maps $F$ with a \textbf{marked critical point} $a = a_F$ of period $p \geq 1$, and a \textbf{free critical point} $-a = -a_F$. If $v = F(a)$ is the corresponding critical value, then Milnor obtained the normal form
\[
F(z) = F_{a,v}(z) = z^3 - 3a^2z + (2a^3 + v)~.
\]

We identify $\mathcal{S}_p$ with the smooth affine curve consisting of all pairs $(a,v) \in \C^2$ such that the marked critical point $a$ is periodic of minimal period $p$ under iteration of $F_{a,v}$. In order to reference the marked critical point or corresponding critical value of a cubic map $F$, denote them by $a_F$ and $v_F$ respectively (the subscript is suppressed if the map referenced is clear from context).

In this notation, the \textbf{filled Julia set} of $F$ is
\[
K_F = \{z \in \C : \, \text{$\{F^n(z)\}$ is bounded}\}
\]
and the \textbf{Julia set} $J_F = \partial K_F$. The \textbf{connectedness locus} $\mathscr{C}(\mathcal{S}_p)$ of $\mathcal{S}_p$, is defined as either
\[
\mathscr{C}(\mathcal{S}_p) = \{F \in \mathcal{S}_p :\, \text{$K_F$ is connected}\}
\]
or, equivalently, as
\[
\mathcal{S}_p \smallsetminus \bigsqcup_h \mathcal{E}_h
\]
where $\mathcal{E}_h$ are the disjoint \textbf{escape regions} of $\mathcal{S}_p$, i.e., the open regions where pairs $(a,v)$ correspond to cubic maps such that the orbit of the free critical point $-a$, and equivalently the orbit of the corresponding point $2a_F$, is unbounded. The point $2a_F$ is called the \textbf{free co-critical point of $F$}, and it has the same image as $-a_F$.

Following the study of the similarity at the Misiurewicz points of the Mandelbrot set, one can redefine the connectedness locus of the curve $\mathcal{S}_p$ as, 
\[
\mathscr{C}(\mathcal{S}_p) = \big\{F \in \mathcal{S}_p :\, \text{$\{F^n(2a_F)\}$ is bounded}\big\} = \big\{F \in \mathcal{S}_p :\, \text{$\{F^n(-a_F)\big\}$ is bounded}\}
\]

More generally, the Julia set of a cubic polynomial is connected if the orbits of both critical points $a_F$ and $-a_F$ are bounded, but since the marked critical point $a_F$ is periodic by definition, connectedness depends only on the orbit of the free critical point $-a_F$. 

In analogy with the quadratic family, we define a \textbf{Misiurewicz map} $F \in \mathcal{S}_p$ as: 

\begin{definition}\label{d-2}
A map $F = F_{a,v}$ is a \textbf{Misiurewicz map} if there exists an $\ell \geq 1$ and a minimal $m \geq 1$ such that
\[
F^\ell_{a,v}(-a_F) =  F^\ell_{a,v}(2a_F) = a_0
\]
and
\[
 F^m_{a,v}(a_0) = a_0
\]
with
\[
\big|(F^m_{a,v})'(a_0)\big| > 1 ~.
\]
The \textbf{preperiod} of such a map is defined as the ordered pair $(\ell, m)$.
\end{definition}

A major difference between the study of Misiurewicz maps in  the curve $\mathcal{S}_p$ and  Misiurewicz points in the Mandelbrot set $\mathbb{M}$ is that the filled Julia set at a Misiurewicz point in $\mathbb{M}$ has empty interior, while the Misiurewicz maps in $\mathcal{S}_p$ must have an attracting cycle, and hence a non-trivial interior, see Figure 2. Thus, we want to look at the filled Julia set of a Misiurewicz map instead of just the Julia set, as was the case in the study of $\mathbb{M}$.

\begin{figure}[h]
    \centering
    \includegraphics[scale=.4]{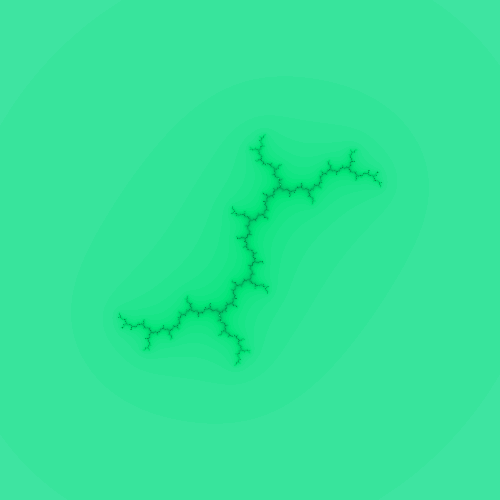}
    \includegraphics[scale=.125]{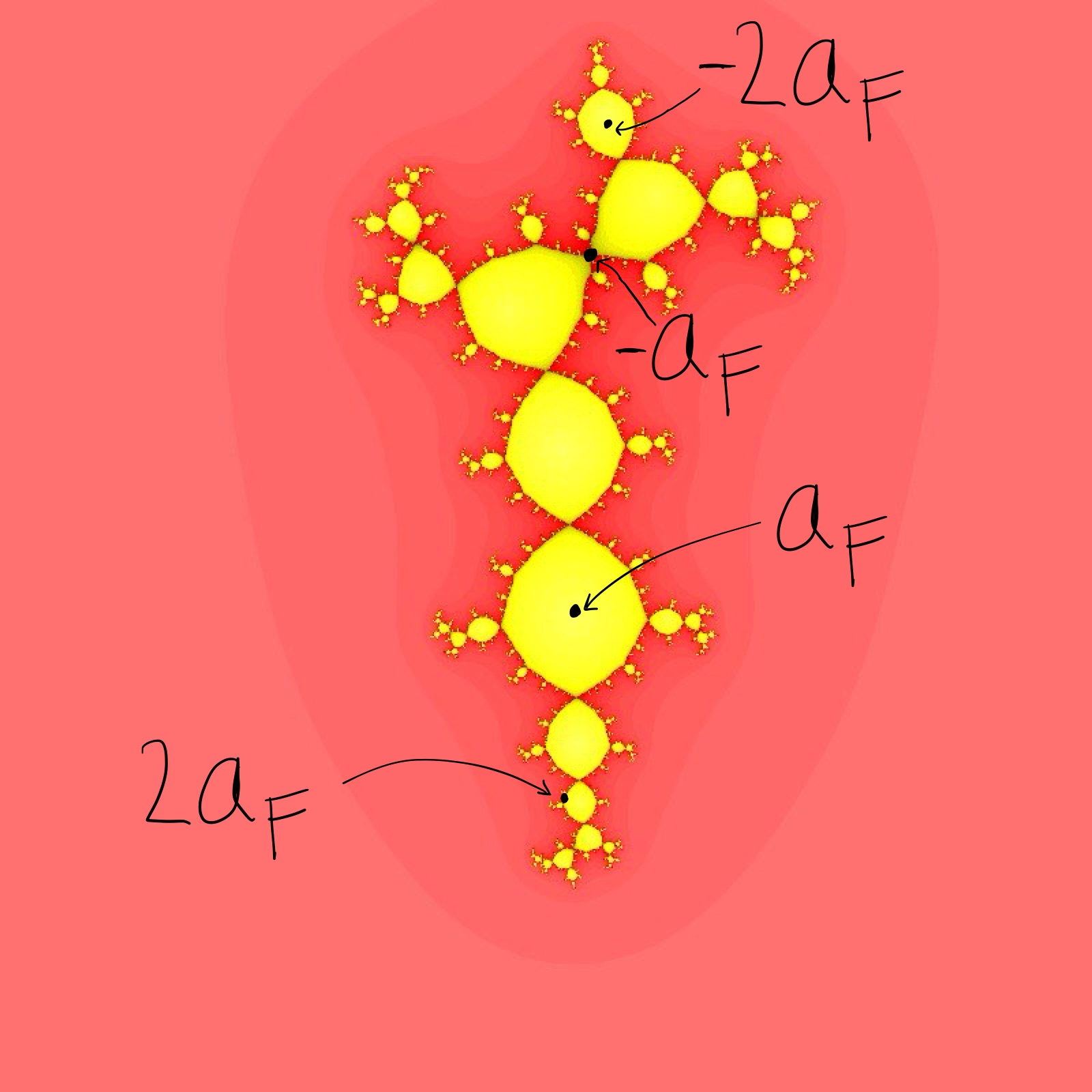}
    \caption{\textsf{\textbf{Left}: The Julia set of the Misiurewicz map $z^2 - i$. \textbf{Right}: The filled Julia set of a Misiurewicz map in $\mathcal{S}_2$, with the critical and co-critical points labeled.}
    \label{fig:enter}}
\end{figure}


\vspace{.25in}

The difficulty in the study of  similarity in the cubic case comes from the dimension. In the quadratic case, elements of $\mathbb{M}$ and $J_c$ can be identified with a point $z_0 \in \C$. This is true for Julia sets of cubic maps as well, but the connectedness locus $\mathscr{C}(\mathcal{S}_p)$ of the curve $\mathcal{S}_p$ is identified by points $(a,v) \in \C^2$. Thus, as discussed above, a local parametrization of $\mathcal{S}_p$ is required to study the local similarity behavior of $\mathcal{S}_p$ near a Misiurewicz map. 

\setcounter{equation}{0}
\setcounter{lem}{0}

\section{Similarity between $\mathscr{C}(\mathcal{S}_p)$ and $K_F$ for a Misiurewicz Map}

In general, as stated in \cite{CM2:2010}, there is no global single variable parametrization of the curve $\mathcal{S}_p$. There are some exceptions for lower periods, as $\mathcal{S}_1$ and $\mathcal{S}_2$ can be parametrized onto $\C$. For example,  if $a_F$ is a fixed point,  $F \in \mathcal{S}_1$, hence the critical value $v_F$ must be equal to the critical point $a_F$. The map
\[
F \in \mathcal{S}_1 \,\to\, (a_F,\, a_F) \in \C \times \C
\]
is its associated coordinate on the curve, and so the curve is isomorphic to $\C$ by the natural projection map. If $a_F$ is periodic of minimal period two,  $F \in \mathcal{S}_2$, then there is a series of affine coordinate changes one can perform to obtain a variable $\delta$ (see \cite{CM1:2009} or \cite{CM4}). Under this change of coordinates the marked critical point $a = -(\delta+1/\delta)/3$, and  $F \in \mathcal{S}_2$  takes the form:
\[
F(z) = z^3 - 3\left(\frac{\delta + 1/\delta}{3} \right)^2 z \,+\, \left(-2\left(\frac{\delta + 1/\delta}{3} \right)^3 \, -\,\left(\frac{\delta + 1/\delta}{3} \right)\right)
\]

\begin{figure}[h]
    \centering
    \includegraphics[scale=.25]{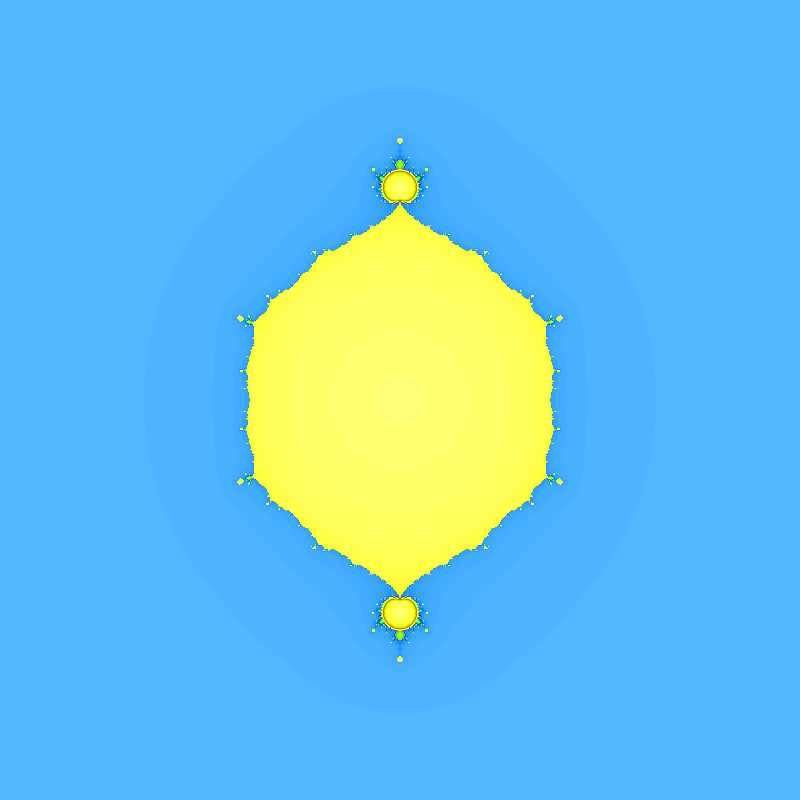}
    \includegraphics[scale=.25]{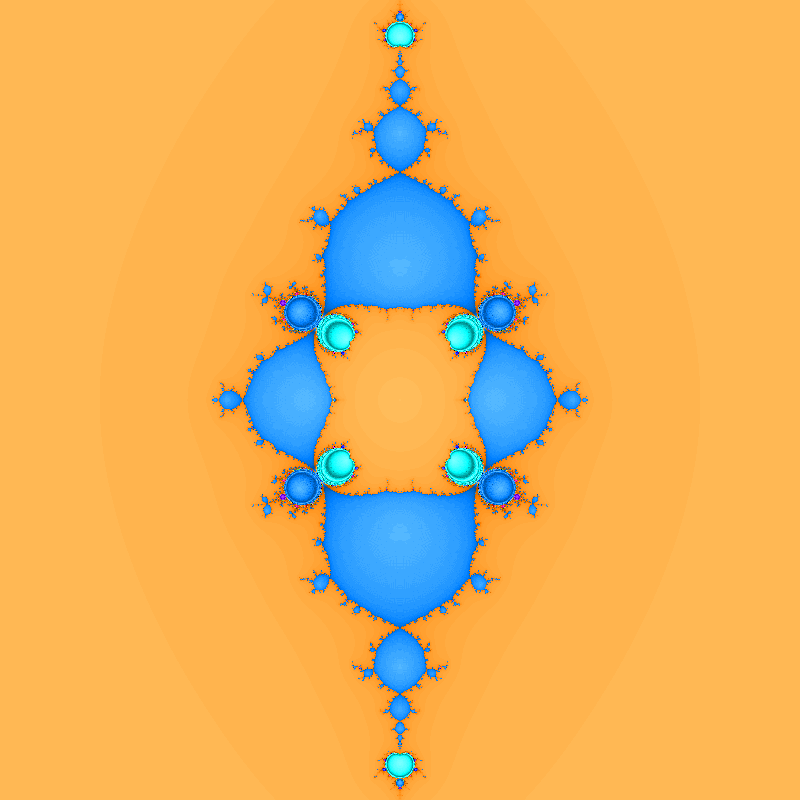}
    \caption{\textsf{The canonical projections of $\mathcal{S}_1$ and $\mathcal{S}_2$ into $\C$ and $\C^*$ respectively.}
    \label{fig:enter}}
\end{figure}

Finally, the curve $\mathcal{S}_3$ has a \textbf{torus coordinate} from an affine coordinate on the universal covering of $\overline{\mathcal{S}_3}$, as described in \cite{CM4}. This is a parametrization accounting for the 8 punctures of $\mathcal{S}_3$. 

In \cite{CM3}, the authors remark that for $p>2$ the \textbf{canonical coordinate} can only be defined as a local coordinate, since it ramifies at most points. (For more details see \cite{CM4}.) 


For this reason, we will use the local parametrization of $\mathcal{S}_p$ introduced in \cite{CM2:2010}. Let $\eta: U \to \C$ with $\eta(z_1, z_2)$  a holomorphic map without critical points throughout some neighborhood $U$ of $\mathcal{S}_p$, such that $\eta|_{\mathcal{S}_p} \equiv 0$. As such, $\bm{t}$ is well defined near any point of $\mathcal{S}_p$ up to translation in the $\bm{t}$-plane by the Hamiltonian system of differential equations
\begin{eqnarray*}
\dfrac{dz_1}{d\bm{t}} &=& \dfrac{\partial \eta}{\partial z_2}\\
\dfrac{dz_2}{d\bm{t}} &=& -\dfrac{\partial \eta}{\partial z_1}~.
\end{eqnarray*}

The natural choice for $\eta$ is
\[
F \mapsto \eta(a_F,v_F) = F^p(a_F) - a_F~.
\]
From this, it follows that all maps $G \in \mathcal{S}_p$ near a given Misiurewicz map $F$ can be represented by a complex number $\bm{t}$. The uniqueness up to translation means that the Misiurewicz map $F$ can correspond to the local coordinate $\bm{t} = 0$. The local parametrization will allow for single variable perturbations in the parameter plane, and, as such, allows for simpler local calculations near the Misiurewicz map in $\mathcal{S}_p$.

\vspace{.25in}

Using this idea, let $F_0 = F$ be a Misiurewicz map in $\mathcal{S}_p$. 

By definition, we have that $F_0^\ell(2a_F) = a_0$, and $F_0^{m}(a_0) = a_0$ with $|(F_0^{m})'(a_0)| > 1$. By the Implicit Function Theorem, there exists a neighborhood $U$ of $0$ in the $\bm{t}$-plane and a holomorphic map $s(\bm{t})$ on $U$ such that $s(0) = a_0$, $s(t) = F_{\bm{t}}^m (s(\bm{t}))$, and $|(F_{\bm{t}}^{m})'(s(\bm{t}))| > 1$.


Within the domain of definition of the local parameter, the repelling periodic point for maps in $U$ is well defined. (One can call $U$ the intersection of the open set from the Implicit Function Theorem and the domain of definition of $\bm{t}$.) The upshot is that in the proof of Theorem~\ref{t-main}, we can zoom in until our window is contained within this open set $U$. This is in contrast with the case of the Mandelbrot set, where the critical value $c$ acts as a global parameter that requires restriction in order to get the continuously defined repelling periodic point. However, in $\mathcal{S}_p$, the local parameter is sufficient.


\subsection*{Proof of Theorem~\ref{t-main}}
The proof of the Main Theorem will follow from the following Lemma:
\begin{lem}\label{l-1}
Suppose that $F = F_{0}$ is a Misiurewicz map of preperiod $(\ell,m)$ with $\ell, m > 0$. For $k \in \N$, set
\begin{equation}\label{e-1}
\rho_k  = \frac{1}{(F^{\ell+km})'(2a_F)}
\end{equation}

Then we have the following:
\begin{itemize}
\item[$(a)$] 
The function $\phi_k(w) = F^{\ell+km}(2a_F + \rho_kw)$ converges uniformly on any compact set to a non-constant entire function $\phi: \C \to \C$ as $k \to \infty$.
\item[$(b)$]
There exists a constant $Q \neq 0$ such that the sequence of functions
\[
\Phi_k(w) = F_{Q\rho_kw}^{\ell+km} \big(2a(Q\rho_k w)\big)
\]
converges to the same function $\phi$ as $k \to \infty$ uniformly on compact sets.
\end{itemize}
\end{lem}


The Lemma relies on the existence of the non-constant entire function $\phi$, referred to in \cite{Kawa20:2020} as a \textbf{Poincar\'e function}. The following theorem describes how $\phi$ is constructed

\begin{thm}
Let $g: \C \to \C$ be an entire function with $g(0) = 0$, $g'(0) = \lambda$, and $|\lambda| > 1$. Then the sequence
\[
\psi_n(w) = g^n(w/\lambda^n)
\]
converges uniformly on compact sets in $\C$. Moreover, the limit function $\psi: \C \to \C$ satisfies 
\[
\text{$(g \circ \psi)(w) = \psi(\lambda w)$ \hspace{.05in}and \hspace{.05in}$\psi'(0) = 1$}
\]
\end{thm}
Once we choose an appropriate function $g$, the entire function $\phi$ can be constructed.

\begin{proof}[\textbf{Proof of Lemma~\ref{l-1}}]
As above, $F$ is a Misiurewicz map with preperiod $(\ell, m)$. Let
\[
g(w) = F^m(a_0 + w) - a_0~.
\]
It is the case that $g(0) = F^m(a_0) - a_0 = 0$, and that $|g'(0)| = |(F^m)'(a_0)| = |\lambda_0| > 1$ by construction. Therefore, we have that
\[
\psi_n(w) = g^n(w/\lambda_0^n) 
\]
converges uniformly in compact sets to a limit function $\psi^*: \C \to \C$. Finally, note that
\begin{equation}\label{e-2}
w \mapsto F^{km}\left(a_0 + \frac{w}{\lambda_0^k}\right) = g^k\left(\frac{w}{\lambda_0^k}\right) + a_0
\end{equation}
and since we have uniform convergence to $\psi^* + a_0$ as $k \to \infty$, let $\phi(z) = \psi^*(z) + a_0$. Hence, Equation~(\ref{e-2}) converges to $\phi$ uniformly on compact sets.
First, we establish the convergence stated in Lemma~\ref{l-1} item $(a)$. Set
\[
A_0 = (F^\ell)'(2a_F) \neq 0
\]
and from Equation~(\ref{e-1}) it follows that
\begin{equation}\label{e-2a}
(F^{\ell+km})'(2a_F) = A_0\lambda_0^k = \frac{1}{\rho_k} ~.
\end{equation}
For sufficiently small $\delta \in \C$,
\begin{equation}\label{e-3}
F^\ell(2a_F + \delta) = a_0 + A_0 \cdot \delta + o(\delta)~.
\end{equation}
Fix an arbitrarily large compact set $E \subset \C$, and take any $w \in E$. Now, set
\[
\delta = w/(A_0\lambda_0^k)~.
\]
The notation $\sim$ below means that the difference between the two elements of the sequence converges to $0$ uniformly on compact sets. From Equation~(\ref{e-3}), it follows that 
\[
F^{km}\left(a_0 + \frac{w}{\lambda_0^k}\right) \sim F^{\ell+km}\bigg(2a_F + \frac{w}{A_0\lambda_0^k} + o(\lambda_0^{-k})\bigg)
\]
and, since $|\lambda_0| >1$
\[
F^{\ell+km}\bigg(2a_F + \frac{w}{A_0\lambda_0^k} + o(\lambda_0^{-k})\bigg) \sim F^{\ell+km}(2a_F + \rho_k w) = \phi_k(w)
\]
as $k \to \infty$. Therefore, we obtain that
\[
\phi(w) = \lim_{k \to \infty} \phi_k(w)~.
\]

Now, it needs to be shown that the sequence in  Lemma~\ref{l-1} item $(b)$ converges to this same entire function $\phi$. Suppose $Q \in \C^*$ is a constant and set
\begin{equation}\label{e-4}
c = c(w) = Q\rho_kw~.
\end{equation}
For these purposes, we denote by $G$ the function
\begin{eqnarray}
G &= &F_{c},  \quad \textrm{so} \nonumber\\
\Phi_k(w) &= &G^{\ell+km}\big(2a(c)\big),   \quad \textrm{and} \label{e-5a}\\
b(c) & = &G^\ell\big(2a(c)\big)~. \label{e-5b}
\end{eqnarray}
where $a(c)$ is the marked critical point $a_G$.

Therefore, for $c$ as in Equation~(\ref{e-4}), denote by $s(c)$ the repelling periodic point of $G$ of period $m$ with $s(0) = a_0$ defined as in Definition~\ref{d-2}, and $\lambda(c)$ denotes its multiplier
\[
\lambda(c) = |(G^m)'(2s(c))|~.
\]

The Poincar\'e function theorem implies that the sequence
\[
\phi_k^c(w) = G^{km}\bigg(s(c) + \frac{w}{\lambda(c)^k}\bigg)
\]
converges uniformly to an entire function $\phi^c(w)$ on compact sets. This can be shown by an identical calculation to the Poincar\'e function calculation in the proof of Lemma~\ref{l-1} item $(a)$.
The map
\[
c ~\mapsto~\phi^c(w)
\]
is holomorphic near $c = 0$. 


Note that at $c=0$, we get that
\[
\phi_k^{0}(w) = G^{km}\bigg(s(0) + \frac{w}{\lambda(0)^k}\bigg) = F^{km}\bigg(a_0 + \frac{w}{\lambda_0^k}\bigg) = \phi_k(w)
\]
which means that if we show that the $\Phi_k \to \phi^c$, then its convergence to $\phi(w)$ is shown.


The remainder of the argument requires transversality with respect to the local parameter $\bm{t}$ about the root $c(w) = 0 \to w = 0$ (since $Q$ and $\rho_k$ are non-zero), which we denote simply as $c = 0$. See Section 5 for details.

By construction, and Corollary~\ref{c-1}, the equation
\[
b(c) - s(c) = 0
\]
has a simple root at $c = 0$, and as such there exists $B_0 \neq 0$ such that we can linearize at $c=0$ as
\[
b(c) - s(c) = B_0c + o(c)~.
\]


Therefore, from Equation~(\ref{e-2a}) and Equation~(\ref{e-4}) we have that
\begin{eqnarray}
b(c) & =& s(c) + B_0\cdot Q\rho_k w + o(\rho_k) \nonumber \\
&=& s(c) + \frac{B_0Q}{A_0}\cdot \frac{\lambda(c)^k}{\lambda_0^k} \cdot \frac{w}{\lambda(c)^k} + o(\rho_k)~. \label{e-6}
\end{eqnarray}

Now, since $\lambda(c)$ is a holomorphic function of $c$, and thus $\lambda(c) = \lambda_0 + O(c)$, we have that
\[
\left|\frac{\lambda(c)}{\lambda_0} - 1\right| = O(c)~.
\]
This implies that
\[
\log\bigg(\frac{\lambda(c)^k}{\lambda_0^k}\bigg) = k \cdot O(c) = O\bigg(\frac{k}{\lambda_0^k}\bigg) \to 0 \qquad \textrm{as} \quad k \to \infty~.
\]
 Hence,
\[
\frac{\lambda(c)^k}{\lambda_0^k} \to 1 \qquad \textrm{as} \quad k \to \infty~.
\]

From Equations~(\ref{e-5a}) and (\ref{e-5b}); and setting $Q=A_0/B_0$ in Equation~(\ref{e-6}), we obtain
\[
\Phi_k(w) = G^{km}(b(c)) = G^{km}\bigg(s(c) + \frac{w}{\lambda(c)^k}+ o(\rho_k)\bigg)
\]
since the right hand side of this equation converges uniformly  to $\phi^c(w)$, as $k \rightarrow \infty$, we get
\[
\lim_{k\to \infty}\Phi_k(w) = \phi^c(w)~.
\]
Further, since
\[
\lim_{c\rightarrow 0}\phi^c(w) = \phi(w)
\]
uniformly on compact sets, we conclude that
\[
\lim_{k \to \infty} \Phi_k(w) = \phi(w)
\]
uniformly on any compact set.  \qquad \end{proof}

\begin{proof}[\textbf{Proof of Theorem~\ref{t-main}}]
Consider the two convergences separately: First, the filled Julia set convergence.
\subsubsection*{Dynamical Space Convergence.}
Let $\phi_k, \,\phi, \,\rho_k$ be defined as in the proof of Lemma~\ref{l-1}.

Since the filled Julia set is a complete invariant of the map $F$, it is the case that
\[
[\rho_k^{-1}(K_F-2a_F)]_r = [\phi_k^{-1}(K_F)]_r
\]


By definition, we have that
\[
[\phi^{-1}(K_F)]_r = [\mathscr{K}]_r
\]
and since, by the Lemma~\ref{l-1}, $\phi_k \to \phi$ uniformly on compact sets, in particular on $\overline{\mathbb{D}(r)}$, then
\[
[\phi_k^{-1}(K_F)]_r \to [\mathscr{K}]_r
\]
as $k \to \infty$, as was to be shown.

\subsubsection*{Parameter Space Convergence.}

Let $q = 1/Q$, ($Q \neq 0$) and $G$ as in the proof of the Lemma~\ref{l-1}. Let
\begin{equation}\label{e-8a}
C_k = \rho_k^{-1}q(\mathscr{C}(\mathcal{S}_p) - F)
\end{equation}
be the shifted, scaled, and locally parametrized connectedness locus of $\mathcal{S}_p$ centered at the Misiurewicz map $F$.

Fix $\epsilon > 0$. Since 
\[
\overline{\mathbb{D}(r)} \smallsetminus N_\epsilon(\mathscr{K})
\]
is compact, there exists an $N = N(\epsilon)$ such that
\[
|F^N \circ \phi(w)| > M
\]
for any $w \in \overline{\mathbb{D}(r)} \smallsetminus N_\epsilon(\mathscr{K})$ and for any $M$. Since by item $(b)$ of Lemma~\ref{l-1}, $\Phi_k(w)$  converges to $\phi(w)$ uniformly on compact sets (namely $\overline{\mathbb{D}(r)} \smallsetminus N_\epsilon(\mathscr{K})$), we get
\[
\Big|G^{(\ell+km)+N}\big(2a(Q\rho_k w\big)\Big| > M
\]
for sufficiently large $k$, as the parameter on $G$ approaches the Misiurewicz parameter $\bm{t}=0$ as $k \to \infty$. Here, the function $a(\bm{t})$ is the projection onto the marked critical point for the map with local parameter $\bm{t}$, as in $a(\bm{t}) = a_{F_{\bm{t}}}$.


This means that the orbit of the co-critical point 
\[
2a(Q\rho_k w)
\]
of the map $G = F_{Q\rho_k w}$ is unbounded, and hence $G$ is not in the connectedness locus for the arbitrary $w \in \overline{\mathbb{D}(r)} \smallsetminus N_\epsilon(\mathscr{K})$. 

In the $\bm{t}$-plane with the given parametrization, and then shifted such that $F$ is at the origin, the local coordinate $w$ (shifted and scaled accordingly) corresponds to the map $G$, which is outside of the connectedness locus as discussed above.

This implies that
\[
w \not\in C_k
\]
as was to be shown, so
\[
[C_k]_r \subset N_\epsilon([\mathscr{K}]_r)~.
\]

Now, the inclusion
\[
[\mathscr{K}]_r \subset  N_\epsilon([C_k]_r)
\]
is to be shown.

Approximate $[\mathscr{K}]_r$ by a finite subset $E$ of $[\mathscr{K}]_r$ such that the $\epsilon/2$ neighborhood of $E$ covers $[\mathscr{K}]_r$. It is enough to prove that for any $w_0 \in E$, there is a sequence
\[
\{w_k\} \subset [C_k]_r
\]
such that
\[
|w_0 - w_k| < \epsilon/2
\]
for sufficiently large $k$.\begin{footnote}{
This is the case because given $w \in [\mathscr{K}]_r$, an $\epsilon/2$ neighborhood of $E$ covers $[\mathscr{K}]_r$ by compactness. Therefore, there must be a $p_w\in E$ such that
\[
|w - p_w | < \epsilon/2~.
\]
Now, if the sequence above exists for $p_w$, then any point in the sequence $w_k$ has the property that
\[
|w - w_k| \leq |w - p_w| + |p_w - w_k| = \epsilon/2 + \epsilon/2 = \epsilon
\]
and so $w \in N_\epsilon([C_k]_r)$.}

\end{footnote}

Let $\Delta$ be a disk of radius $\epsilon/2$ centered at $w_0$. We may assume that $\Delta \subset \mathbb{D}(r)$. If not, then $\Delta \cap \partial\mathbb{D}(r) \neq \emptyset$, we can take $w_k \in \partial\mathbb{D}(r)$ to be the sequence in question. By definition,
\[
[C_k]_r = \big(C_k \cap \mathbb{D}(r)\big) \cup \partial \mathbb{D}(r)
\]
the sequence $\{w_k\}$ being on the boundary of the disk of radius $r$ still gives a desired sequence in $[C_k]_r$, and the containment would hence be trivial.

Since $\phi(w_0) \in K_F$, as $E \subset \mathscr{K} = \phi^{-1}(K_F)$, consider the case where $\phi(w_0) \in J_F$. If not, then for large enough $k$, 
\[
[C_k]_r = \mathbb{D}(r)
\]
and the sequence is simple to construct.

If $\phi(w_0) \in J_F$, then since repelling cycles are dense in $\partial K_F = J_F$, we can choose $w_0'$ such that $\phi(w_0')$ is a repelling periodic point of some period $d$ and
\[
|w_0 - w_0'| < \epsilon/4~.
\]

So we have that the function
\[
\chi: w \mapsto F^d\big(\phi(w)\big) - \phi(w)
\]
has a zero at $w = w_0'$.

Now, consider the function
\[
\chi_k : w \mapsto G^d\big(\Phi_k(w)\big) - \Phi_k(w)
\]
which converges to $\chi$ since $\Phi_k \to \phi$ as $k \to \infty$.

Hurwitz' Theorem says that for a sequence of holomorphic functions that converge to a non-constant entire function, if the limit function has a zero of order $d$, then in a sufficiently small neighborhood of radius $\rho > 0$ of the root and for sufficiently large $k \in \N$, the functions $f_n$ with $n \geq k$ have precisely $d$ roots in the disk
\[
|w - w_0| < \rho
\]
which, in our case can be made to
\[
|w - w_0| < \epsilon/4
\]
or smaller.

Therefore, since $\chi$ has a root at $w_0'$ and each of the $\chi_k$ is a holomorphic function converging to a non-constant entire function $\chi$, by the uniform convergence of $\Phi_k(w)$ to $\phi(w)$,  there must be a root of $\chi_k$, denoted $w_k$, in the disk with
\[
|w_k - w_0'| < \epsilon/4
\]
for all large $k$.


Therefore, for a particular
\begin{equation}\label{e-9}
c_k = c(w_k) = Q\rho_kw_k
\end{equation}
we get,
\begin{equation}\label{e-10}
G^d\big(\Phi_k(w_k)\big) = \Phi_k(w_k)
\end{equation}
since $w_k$ is a solution to $\chi_k$ by construction. In addition, it follows from Equation~(\ref{e-10}) that
\[
F_{c_k}^{d+(\ell+km)}\big(2a(Q\rho_kw_k)\big) = F_{c_k}^{\ell+km}\big(2a(Q\rho_kw_k)\big)
\]
or, more succinctly using Equation~(\ref{e-9})
\[
F_{c_k}^{d+(\ell+km)}\big(2a(c_k)\big) = F_{c_k}^{\ell+km}\big(2a(c_k)\big)
\]
From this, it is clear that both critical points of $F_{c_k}$ have bounded orbit, and so  $c_k \in \mathscr{C}(\mathcal{S}_p)$. As such, it is the case that $w_k \in C_k$, with $|w_k - w_0| < \epsilon/2$ by the triangle inequality. Specifically, we have that
\[
|w_k - w_0| \leq |w_k - w'_0| + |w'_0 - w_0| = \epsilon/4 + \epsilon/4 =\epsilon/2
\]
\end{proof}




\setcounter{lem}{0}
\section{Transversality of Repelling Periodic Points}
This section functions as an appendix, where we establish the transversality needed in the Proof of Lemma~\ref{l-1}.

More precisely, we will establish the transversality of $b(c)$ and $s(c)$ which appear in the proof of Lemma~\ref{l-1}. In \cite{ON} and \cite{DPM},  it is shown that in the quadratic case, the map
\[
c \, \mapsto \, b(c) - s(c)
\]
has a simple zero, allowing the linearization of $b(c)-s(c)$, and so a way to sidestep calculating large iterates of small perturbations of $F_{\bm{t}}$ around $\bm{t} = 0$. This fact needs to be proven. 

Note that, the notation unravels to saying that
\[
G^{\ell}\big(2a(c)\big) -s(c) = F_{c}^\ell\big(2a(c)\big) - s(c)
\]
has a simple root at $c = 0$. It is clear that $c = 0$ is a root, as
\[
F_0^\ell\big(2a(0)\big) - s(0) = a_0 - a_0 = 0
\]
It needs to be shown that this zero is simple. The external ray method from \cite{ON} can be used, but it requires the notions of dynamical and parameter rays in $\mathcal{S}_p$, which are studied thoroughly in $\cite{CM3}$.

\subsection*{External Rays in $\mathcal{S}_p$}

The following definition of external dynamical rays is given in \cite{CM3}.

For the cubic polynomial $F$ of Equation~(\ref{e-0}), the dynamic Green's function \break $\bm{g}_F : \C \to \R$ is defined as,
\[
\bm{g}_F(z) = \lim_{n\to \infty} \dfrac{1}{3^n} \log |F^n(z)|~.
\]

The \textbf{external dynamic ray} $\mathscr{R}(K_F,\theta)$ of angle $\theta$ is an orthogonal trajectory to the family of equipotentials $\bm{g}_F(z) =C$. For $F $ in the space of cubic polynomials $P_3$ and $\theta \in \mathbb{T} = \R/\Z$, if $\mathscr{R}(K_F, \theta)$ does not bifurcate, we define
\[
\bm{r}_{F,\theta}: \R_+^* \to \mathscr{R}(K_F,\theta)
\]
by
\[
\bm{g}_F\big(\bm{r}_{F,\theta}(s)\big) = s~.
\]
If it bifurcates on a critical point $\omega$ of $\bm{g}_F$ (which are the critical and precritical points of $F$), the function $\bm{r}_{f,\theta}$ is only defined on $[\bm{g}_F(\omega), +\infty]$. If $\mathscr{R}(K_F, \theta)$ lands at a point $\alpha$, we extend $\bm{r}_{F,\theta}$ to $\R_+$ by setting $\bm{r}_{F,\theta}(0) = \alpha$.


We need an equivalent statement to Theorem 8.2 in \cite{ON} for the curve $\mathcal{S}_p$. In essence, we want
\begin{thm}\label{t-app1}
Let $F \in \mathscr{C}(\mathcal{S}_p)$ be a Misiurewicz map, i.e. the critical point $2a_F$ is repelling strictly preperiodic. Let $\mu_h \geq 1$ be the multiplicity of the bordering escape region $\mathcal{E}_h$.  Then:
\begin{enumerate}
\item 
The point $2a_F$ has a finite number of external rays in $K_F$.
\item
For each external argument $\theta$ of $2a_F$ in $K_F$, the external ray $\mathscr{R}\big(\mathscr{C}(\mathcal{S}_p), \,\theta/\mu_h\big)$ lands at $F$.
\end{enumerate}
\end{thm}

This result is known to be true by \cite{CM3}, but we reprove it in a different way in order to get the transversality needed for the proof of Lemma~\ref{l-1}.

In \cite{ON}, it was shown that for all polynomials of degree $d$, in particular $d=3$, we have that the map
\[
(F,s) \mapsto \bm{r}_{F,\theta}(s)
\]
from $\Lambda \times \R_+ \to \C$ is continuous and holomorphic with respect to $F$ for all $\theta \in \R/\Z$. Here $\Lambda$ is a neighborhood of $F$ in $\mathcal{S}_p$, which we will intersect with the domain of definition of the local parameter $\bm{t}$. This means that this continuity is also in a neighborhood of $\bm{t}=0$.

Following \cite[Lemma 5.9]{CM1:2009}, each escape region can be expressed in terms of a  \textbf{Böttcher coordinate}. For $F \in \mathcal{E}_h$ an arbitrary escape region, there is an associated Böttcher map $\mathfrak{B}_F(z)$ defined for all $z \in \C$ with $|z|$ sufficiently large. Such map satisfies
\[
\mathfrak{B}_F\big(F(z)\big) = \mathfrak{B}_F(z)^3
\]
with $|\mathfrak{B}(z)| > 1$, and with $\mathfrak{B}(z)/z \to 1$ as $|z| \to \infty$. In particular, $\mathfrak{B}_F(2a)$ will be defined, so consider
\begin{align*}
\mathfrak{B}: \mathcal{E} &\to \C\setminus \overline{\mathbb{D}}\\
F &\mapsto \mathfrak{B}_{F}(2a)~.
\end{align*}
This notion parametrizes the escape region, and the value $\mathfrak{B}_{F}(2a)$ is the \textbf{Böttcher coordinate} for the map $F$.

Lastly, we define the notion of a parameter ray in an escape region, as given in \cite{CM3}.

\begin{definition}\label{d-3}
The conformal isomorphism $\mathcal{E}_h \to \C \smallsetminus \overline{\mathbb{D}}$ is given by choosing a smooth branch of $F \mapsto \sqrt[\mu]{\mathfrak{B}_F(2a_F)}$, the $\mu$th root of the B\"ottcher coordinate $\mathfrak{B}_F$ evaluated at $2a_F$. The \textbf{parameter angle} is then
\[
\phi(F) = \arg\big(\sqrt[\mu]{\mathfrak{B}_F(2a_F)}\,\big) \in \R/\Z~.
\]

For any angle $\phi_0$, the set of all $F \in \mathcal{E}_h$ with parameter angle $\phi(F) = \phi_0$ is called a \textbf{parameter ray $\mathscr{R}_{\mathcal{E}_h}(\phi_0)$}.
\end{definition}

The first part of the Theorem is clear from the definition of a Misiurewicz map. It needs to be shown that $2a_F$ has finitely many external rays. Recall now that $F^\ell(2a_F)$ is a repelling periodic point of $F$, and hence has finitely many external rays.

Hence, if $2a_F$ had infinitely many external arguments, then $\ell$ iterations of the angle tripling map takes an infinite set to a finite set. However, since the angle tripling map on $\R/\Z$ is a 3 to 1 map, then $\ell$ iterations of the angle tripling map is a $3^{\ell}$ to 1. Since $3^{\ell}$ is finite, the angle tripling map does not map the infinitely many preimage angles into a finite number of external arguments for $F^{\ell}(2a_F)$. Therefore, $2a_F$ cannot have infinitely many external arguments. 

\vspace{.25in}

The proof of the second part of the theorem follows from \cite{ON}, with notable changes, namely that one must specify that there is an escape region $\mathcal{E}_h$ that the parameter ray lives in.

\vspace{.2in}
\begin{proof}[\textbf{Proof of Theorem~\ref{t-app1} part 2.}]
The point $2a_F$ is a repelling preperiodic point of $F$ and we have that $\bm{r}_{F,\theta}(0) = 2a_F$ by the assumption that $\theta$ is an external argument of $2a$ in $K_F$. The point $2a_F$ clearly does not have any critical points in its forward orbit. For $\lambda$ a local parameter near 0 and $s \in \R^+$, set 
\[
H_s(\lambda) = \bm{r}_{F_{\lambda},\theta}(s) - 2a(\lambda)~.
\]
Denote by $\nu$ the order of the zero of $H_0$ at $\lambda = 0$. The order $\nu$ is clearly greater than 0 since $2a_F$ is the landing point when $\lambda = 0$, and $2a(\lambda) =2a(0) = 2a_F$. We have that $\nu < \infty$ since the point in question is an isolated preperiodic point. For $s > 0$ but close to $0$, the Hurwitz Theorem says that
\[
H_s(\lambda) = 0
\]
has $\nu$ solutions close to $\lambda = 0$, up to multiplicity. 


Choose a root $R$, and so
\[
2a(R) = \bm{r}_{F_R, \theta}(s)
\]
It follows that $2a(R) \not\in K_{F_R}$, and hence $F_R \not\in \mathscr{C}(\mathcal{S}_p)$, and in particular $F_R \in \mathcal{E}_h$ for some escape region of $\mathcal{S}_p$. 


The Böttcher coordinate for each $F$ maps into $\C \smallsetminus \overline{\mathbb{D}}$. As such, since the point $2a(R)$ is an element outside of $K_{F_R}$, it can be represented as a point outside of the closed disk in $\C$, specifically as the point
\[
\mathfrak{B}_{F_R}\big(2a(R)\big) = e^{s + 2\pi i \theta}
\]
with $s > 0$.

Now, consider the conformal isomorphism $\mathcal{A}:\mathcal{E}_h \to\C\smallsetminus \overline{\mathbb{D}}$, by choosing a smooth branch of
\[
\mathcal{A}(F_{\bm{t}}) = \sqrt[\mu]{\mathfrak{B}_{F_{\bm{t}}}\big(2a(\bm{t})\big)} ~.
\]
Now, we have the chain of maps
\[
\begin{tikzcd}
\C \smallsetminus K_{F_R} \arrow[r,"\mathfrak{B}_{F_R}"]& \C \smallsetminus \overline{\mathbb{D}} &\mathcal{E}_h\arrow[l,"\mathcal{A}",swap]~.\\
\end{tikzcd}
\]
Consider that
\begin{eqnarray*}
\mathcal{A}(F_{R}) &=& \sqrt[\mu]{\mathfrak{B}_{F_R}\big(2a(R)\big)} = \sqrt[\mu]{\exp(s + 2\pi i \theta)}\\
&=& \zeta_\mu \exp\left(\dfrac{s + 2\pi i \theta}{\mu}\right) = \exp\left(\dfrac{2\pi n i}{\mu}\right)\cdot \exp\left(\dfrac{s + 2\pi i \theta}{\mu}\right)
\end{eqnarray*}
with $\zeta_\mu$ an $n$-th root of unity.  Simplifying the previous equation we get,
\[
\mathcal{A}(F_{R}) = \exp\left(\dfrac{s}{\mu} + \dfrac{2\pi i(n+\theta)}{\mu}\right)
\]
Choosing the smooth branch amounts to the branch where $n = 0$, as this branch preserves continuity in $\mathcal{E}_h$ near $F$. Therefore, 
\[
\mathcal{A}(F_{R}) = \exp\left(\dfrac{s}{\mu} + 2\pi i\dfrac{\theta}{\mu}\right)~.
\]

As $s \to 0$, the parameter ray will land on the boundary of the escape region. This implies that the parameter ray with angle $\theta/\mu$ will land at $F$ as $s \to 0$, which is as $\lambda \to 0$, as was to be shown.
\end{proof}

The formulation of the corollary is closer to the one found in \cite{DPM}, but for the proof, compare with \cite{ON}.

\begin{cor}\label{c-1}
The equation
\[
b(c) - s(c) = 0
\]
has a simple zero at $\lambda = 0$.
\end{cor}

\begin{proof}
The multiplicity of $0$ as a root of the equation above is equal to $\nu$ in Part 2 of the Proof of Theorem~\ref{t-app1}.

For $s > 0$, the equation 
\[
H_s(\lambda) = 0
\]
has only one solution, since it is necessarily $\mathcal{A}^{-1}(e^{(s +2\pi i\theta)/\mu_h})$ by the previous part. 


Map the equation
\[
2 a(\lambda) = \bm{r}_{F_{\lambda}, \theta}(s)
\]
to the equation
\[
\mathfrak{B}_{F_\alpha}\big(2a(\lambda)\big) = e^{(s +2\pi i\theta)/\mu_h}
\]
via the diffeomorphism
\[
(\lambda, z) \mapsto (\lambda, \mathfrak{B}_{F_\alpha}(z))
\]
The multiplicity of the solution $2 a(\lambda) = \bm{r}_{F_\alpha, \theta}(s)$ is equal to the multiplicity of the solution to
\[
\mathcal{A}(F_\lambda) = e^{(s +2\pi i\theta)/\mu_h}
\]
which is one since $\mathcal{A}$ is a conformal isomorphism by construction, hence bijective.
\end{proof}

\bibliographystyle{alpha}
\bibliography{refs}

\end{document}